\documentclass[reqno]{amsart}

\usepackage{amssymb}

\usepackage{amsmath}

 \usepackage{amsthm,enumerate}
 
 \usepackage{tikz}

\newtheorem{definition}{Definition}
\newtheorem{theorem}{Theorem}
\newtheorem{lemma}{Lemma}
\newtheorem{prop}{Proposition}

\newcommand{\Z}{\mathbb{Z}}
\newcommand{\floor}[1]{\lfloor #1\rfloor}

\begin{document}

\title[Breaks of log-concavity in the independence polynomials of trees]{Multiple breaks of log-concavity in the independence polynomials of trees}

\author{C\'esar Bautista-Ramos} 
\email{cesar.bautista@correo.buap.mx}

\address{Facultad de Ciencias de la Computaci\'on, Benem\'erita Universidad Aut\'onoma de Pue\-bla, 14 sur y Avenida San Claudio, CC03-304, Ciudad Universitaria, Puebla, 72570, Puebla, Mexico}

\begin{abstract}
We construct infinite families of trees whose independence polynomials violate log-concavity at an arbitrary number of indices. This affirmatively answers a question of D. Galvin.
\end{abstract}



\maketitle

\section{Introduction}
\label{intro}
The conjecture that independence polynomials of trees are unimodal \cite{Alavi} was later strengthened to the more tractable conjecture that they are log-concave \cite{Levit}. Several authors have since disproven this \cite{Kadrawi, Kadrawi2,Galvin,Ramos}. At the time of this writing, all known counterexamples in the literature exhibited only a single break in log-concavity. This observation led Galvin \cite{Galvin} to ask whether multiple breaks are possible.

In this note, we answer that question affirmatively. We construct infinite families of trees whose independence polynomials break log-concavity at an arbitrary number of indices (Theorem~\ref{thm:main}).

\section{Definitions and Results}
We use standard definitions, notations, and concepts from the literature concerning independent sets in simple graphs and the log-concavity of sequences (see, for instance, \cite{Kadrawi, Galvin}). In particular, if $G$ is a graph and $v$ is a vertex of $G$, we use $N[v]$ to denote the closed neighborhood of $v$, while $G\setminus v$ and $G\setminus N[v]$ denote the corresponding graphs with those vertices deleted. We denote the independence polynomial of a graph $G$ in the indeterminate $x$ by $I(G)$. The degree of $I(G)$ is the independence number of the graph, denoted $\alpha(G)$. We also use the standard recursive formulas for computing the independence polynomial \cite{Arocha,Gutman}. For a simple graph $G$ and a vertex $v$ of $G$,
\begin{equation}
\label{eq:Aro}
I(G) = I(G \setminus v) + x \,I(G \setminus N[v]).
\end{equation}
If $G$ is the disjoint union of graphs $G_1$ and $G_2$, then
\begin{equation}
\label{eq:Aro1}
I(G) = I(G_1) I(G_2).
\end{equation}

In the following, we define the trees $TG_{m,t}$ that exhibit as many as 
$m$ violations of the log-concavity of their independence polynomials.

\begin{definition}
Let $m$ and $t$ be two nonnegative integers.
\begin{enumerate}[(i)]
    \item Let $P_m$ denote the path graph on $m$ vertices.
    \item Let $S_{2,m}$ denote the starlike graph with a root $w$ to which $m$ copies of $P_2$ are attached.
    \item The tree $T_{m,t}$ is defined as the tree with a root $v$ that has $m$ children, and each child has $t$ pendant $P_2$ paths.
    \item The tree $TG_{m,t}$ is obtained by taking $m$ disjoint copies of $T_{3,t}$, joining the root $v$ of each copy to a new root $v_0$, and adding one additional child to $v_0$ that has no further descendants.
\end{enumerate}
\end{definition}

The trees $T_{m,t}$ and $S_{2,t}$ were first defined by Galvin \cite{Galvin} in order to prove a conjecture of Kadrawi and Levit \cite{Kadrawi2}. Fig.~\ref{fig1} shows the tree $TG_{2,5}$, whose independence polynomial is non-log-concave, with violations at two indices. In what follows, we explain why.
\begin{figure}[bth]
\centering
\begin{tikzpicture}[
    vertex/.style={circle, fill=black, minimum size=4pt, inner sep=0pt},
    scale=0.9,
    every node/.style={transform shape}
]

    \node[vertex, label=above:$v_0$] (v0) at (0, -1) {};

    \node[vertex] (l0) at (0, -2) {};
    \draw (v0) -- (l0);



    \node[vertex, label=above left:$v$] (v1) at (-4, -2.5) {};
    \draw (v0) -- (v1);

    \foreach \x/\i in {-6/1, -4/2, -2/3} {
        \node[vertex] (c1\i) at (\x, -4) {};
        \draw (v1) -- (c1\i);
        
        \foreach \lx in {-0.8, -0.4, 0, 0.4, 0.8} {
            \node[vertex] (leaf) at (\x+\lx, -5.5) {};
            \node[vertex] (leaf1) at (\x+\lx, -6.5) {};
            \draw (c1\i) -- (leaf);
            \draw (leaf) -- (leaf1);
        }
    }


    \node[vertex, label=above right:$v$] (v2) at (4, -2.5) {};
    \draw (v0) -- (v2);

    \foreach \x/\i in {2/1, 4/2, 6/3} {
        \node[vertex] (c2\i) at (\x, -4) {};
        \draw (v2) -- (c2\i);
        
        \foreach \lx in {-0.8, -0.4, 0, 0.4, 0.8} {
            \node[vertex] (leaf) at (\x+\lx, -5.5) {};
            \node[vertex] (leaf1) at (\x+\lx, -6.5) {};
            \draw (c2\i) -- (leaf);
            \draw (leaf) -- (leaf1);
        }
    }

\end{tikzpicture}
\caption{The tree $TG_{2,5}$, whose independence polynomial has degree 37 and  breaks log-concavity at indices 34 and 36. For $TG_{m,t}$, the parameter $m$ equals the number of log-concavity violations for sufficiently large $t$.}\label{fig1}
\end{figure}
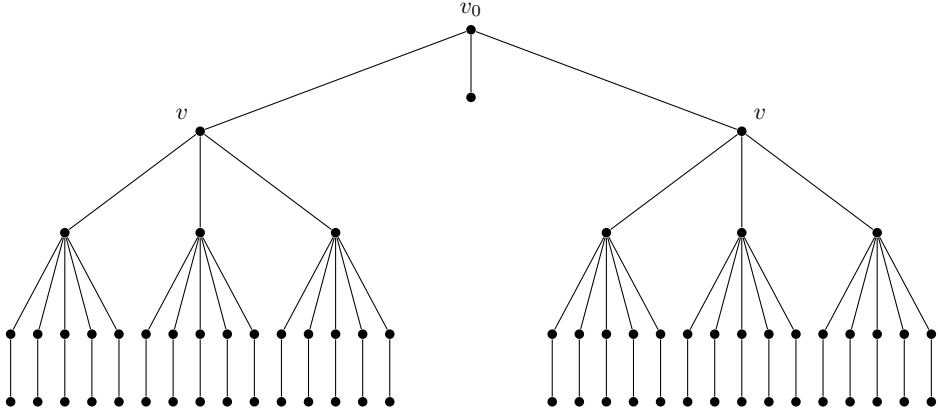

The equations $I(P_1) = 1+x$, $I(P_2) = 1+2x$, together with \eqref{eq:Aro} and \eqref{eq:Aro1}, yield the following.
\begin{lemma}\label{l:degs}
Let $m$ and $t$ be two nonnegative integers. Then,
\begin{equation}\label{eq:ST}
I(S_{2,t})=I(P_2)^t+xI(P_1)^t,\quad I(T_{m,t})=S_{2,t}^m+xP_2^{mt}
\end{equation}
and $\alpha(S_{2,t})=t+1$, $\alpha(T_{m,t})=m(t+1)$.
\end{lemma}	

The trees in \cite{Kadrawi,Kadrawi2,Ramos} have independence polynomials whose log-concavity is broken near the leading coefficient. A similar phenomenon occurs for our trees $TG_{m,t}$. To expose the relevant coefficients, it is convenient to consider the \emph{reflected} independence polynomials.

\begin{definition}
  Let $p \in \Z[x]$ be a polynomial. The \emph{reflected polynomial} of $p$, denoted by $Rp$, is defined as
  \[
  Rp(x) = x^n p(1/x),
  \]
  where $n = \deg p$.
\end{definition}

From the definition, it follows that the reflection operator $R$ is multiplicative: for any polynomials $p, q \in \Z[x]$, we have
\[
R(pq) = (Rp)(Rq).
\]
Additionally, $R$ satisfies an additive property under certain conditions: If $\deg(p+q) = \deg p$, then
\[
R(p+q) = Rp + x^{\deg p - \deg q} Rq.
\]

Applying \eqref{eq:Aro} to the root vertex $v_0$ of $TG_{m,t}$ yields
\begin{equation}\label{eq:ITG}
I(TG_{m,t}) = I(P_1) I(T_{3,t})^m + x I(S_{2,t})^{3m},
\end{equation}
which implies, together with Lemma~\ref{l:degs}, that $\alpha(TG_{m,t}) = 3(t+1)m + 1$. We now define the coefficients of the reflected independence polynomial of $TG_{m,t}$.

\begin{definition}
Let $c_k(m,t) \in \mathbb{Z}$ be such that 
\[
R I(TG_{m,t}) = \sum_{k=0}^{\infty} c_k(m,t) x^k.
\]
\end{definition}

We can compute the first few coefficients $c_k(m,t)$ as follows.
After applying the reflection operator to both sides of \eqref{eq:ITG} and using its additive and multiplicative properties together with Lemma~\ref{l:degs}, we obtain
\begin{equation}\label{eq:main}
R I(TG_{m,t}) = (x+1) R I(T_{3,t})^m + R I(S_{2,t})^{3m}.
\end{equation} 
 Furthermore, from Lemma~\ref{l:degs}, we use the equations
\begin{equation}\label{eq:RTRS}
R I(T_{3,t}) =R I(S_{2,t})^3 + x^{2} (x+2)^{3t},\quad
R I(S_{2,t}) = (x+1)^t + x(x+2)^t,
\end{equation}
where we expand $(x+2)^{t}$ and $(x+1)^t$ using the binomial theorem. However, we are only interested in the highest-order terms. To this end, we use the exact-order $\Theta$ asymptotic notation \cite[p. 87]{Brassard}, \cite[p. 110]{Knuth}.

\begin{definition}
Let $A= (a_k(t))_{k=0}^{\infty}$ and $B= (b_k(t))_{k=0}^\infty$ be sequences where each term is an eventually nonnegative function of the positive integer variable $t$. Let $S_A(x)$ and $S_B(x)$ be their respective generating functions with indeterminate $x$, and let $m$ be a fixed positive integer. We write $S_A(x) \asymp S_B(x) \pmod{x^m}$ if $a_k(t) \in \Theta(b_k(t))$ for all $0 \leq k \leq m-1$.
\end{definition}

The relation $\asymp$ is an equivalence relation because this property is inherited from the exact-order relation $\Theta$. Furthermore, $\asymp$ is consistent with the sum and product of generating functions in the following sense.

\begin{prop}\label{prop:0}
Let $m$ be a fixed positive integer. Let $A,B,C,D$ be four sequences of eventually nonnegative functions of the positive integer variable $t$. If 
\[S_A(x)\asymp S_C(x)\text{ and }S_B(x)\asymp S_D(x)\pmod{x^m}\]
then
\[
S_A(x)+S_B(x)\asymp S_C(x)+S_D(x)\text{ and }S_A(x)S_B(x)\asymp S_C(x)S_D(x)\pmod{x^m}.
\]
\end{prop}

The proof of this proposition is elementary, so we omit it. This proposition is used frequently in the proof of the following lemma.

\begin{lemma}\label{l:comp}
Let $m$ be a fixed positive integer and let $t$ be a positive integer variable. Then:
\begin{enumerate}[(i)]
\item\label{l:i} $(2+x)^t\asymp2^t(1+x)^t\pmod{x^m}$.
\item\label{l:i1} $(2+x)^t(1+x)\asymp (2+x)^t\pmod{x^m}$.
  \item \label{l:i2} $(1+2^tx)^m(1+x)\asymp (1+2^tx)^m\pmod{x^{m+1}}$.
  \item\label{l:ii} $RI(S_{2,t})\asymp 1+2^tx(1+x)^t\pmod{x^m}$.
    \item\label{l:ii1} $RI(S_{2,t})^3\asymp (1+2^tx)^{2}+x^3(2+x)^{3t} \pmod{x^m}$.
\item\label{l:iv} $RI(T_{3,t})\asymp (1+2^tx)^{2}+x^{2}(2+x)^{3t}\pmod{x^m}$.
\item\label{l:v} $RI(T_{3,t})^{m}(x+1)\asymp \sum_{k=0}^{\infty} 2^{(k+\floor{k/2})t}\pmod{x^{2m+1}}$.
\item\label{l:vi} $RI(S_{2,t})^{3m}\asymp \sum_{k=0}^{\infty}2^{kt}x^k\pmod{x^{2m+1}}$.
\end{enumerate}
\end{lemma}

\begin{proof}
\begin{enumerate}[(i)]
\item
From the binomial theorem and using that $\binom{t}{k}\in\Theta(t^k)$ for $0\leq k\leq m-1$, the result follows.

\item Using \eqref{l:i} and  $(1+x)^{t+1}\asymp (1+x)^t\pmod{x^m}$, we obtain the result.
\item From the binomial theorem,
  \begin{align*}
    (1+2^tx)^m(1+x)&\asymp \sum_{j=0}^{m}2^{jt}x^j+\sum_{j=0}^m 2^{jt}x^{j+1}\\
    & \asymp 1+\sum_{j=1}^m(2^{jt}+2^{(j-1)t})x^j+2^{mt}x^{m+1}\\
    &\asymp (1+2^tx)^m\pmod{x^{m+1}}.
  \end{align*}
\item From \eqref{eq:RTRS} and \eqref{l:i} we obtain:
  \begin{align*}
    RI(S_{2,t})&\asymp (1+x)^t(1+2^tx)\asymp \sum_{k=0}^\infty t^kx^k+\sum_{k=0}^\infty 2^t t^k x^{k+1}\\               
               &\asymp 1+\sum_{k=1}^\infty (t^k+2^t t^{k-1})x^k\asymp 1+2^tx(1+x)^t\pmod{x^m}.
  \end{align*}

\item From \eqref{l:ii}, the binomial theorem, and \eqref{l:i}, we obtain:
\begin{align*}
RI(S_{2,t})^3&\asymp 1+x2^t(1+x)^t+x^2 2^{2t}(1+x)^{2t}+x^3 2^{3t}(1+x)^{3t}\\
&\asymp 1+\sum_{k=0}^\infty 2^tt^{k}x^{k+1}+\sum_{k=0}^\infty 2^{2t}t^{k}x^{k+2}+\sum_{k=0}^\infty 2^{3t}t^{k}x^{k+3}\\   
&\asymp 1+2^tx+(t2^{t}+2^{2t})x^2+\sum_{k=3}^\infty (2^{t}t^{k-1}+2^{2t}t^{k-2}+2^{3t}t^{k-3})x^k\\
&\asymp (1+2^tx)^{2}+x^3(2+x)^{3t} \pmod{x^m}.
\end{align*}

\item From \eqref{eq:RTRS}, \eqref{l:ii1}, and \eqref{l:i1}, we obtain:
\begin{align*}
RI(T_{3,t})&\asymp (1+2^tx)^{2}+x^3(2+x)^{3t}+x^{2}(2+x)^{3t}\\
&\asymp (1+2^tx)^{2}+x^{2}(2+x)^{3t}(1+x)\\
&\asymp (1+2^tx)^{2}+x^{2}(2+x)^{3t} \pmod{x^{2m+1}},
\end{align*}
where we work modulo \(x^{2m+1}\) in preparation for the proof of \eqref{l:v}.

\item By \eqref{l:iv} and  the binomial theorem,  we obtain
\begin{align*}
RI&(T_{3,t})^m (1+x)\asymp (1+2^tx)^{2m}(1+x)+\\
&\,\sum_{j=1}^m(1+2^tx)^{2(m-j)} x^{2j}(x+2)^{3jt}(1+x) \pmod{x^{2m+1}}.
\end{align*}
From \eqref{l:i2} and \eqref{l:i1}, we know that the factor $1+x$ can be ignored. Thus, using  \eqref{l:i}, we can write
\begin{align*}
RI(T_{3,t})^m (1+x)&\asymp \sum_{j=0}^m(1+2^tx)^{2(m-j)} x^{2j}(2+x)^{3jt}\\&
\asymp \left(\sum_{k=0}^\infty t^{k} x^k\right)Q \pmod{x^{2m+1}},\\
\end{align*}
where
\[
Q=\sum_{j=0}^m \sum_{k=0}^{2(m-j)}2^{(3j+k)t}x^{2j+k}.
\] In order to find the exact order of the coefficients of $Q$, we make the change of variables
\begin{equation}
  \label{eq:diof}
\ell=2j+k.  
\end{equation} 
We must determine the coefficients of each monomial $x^\ell$. Observe that $\ell + j = 3j + k$, and that the largest integer solution $j$ in \eqref{eq:diof} occurs when $j$ is the quotient of the integer division of $\ell$ by $2$ and $k$ is its remainder, i.e., when $j = \lfloor \ell/2 \rfloor$. Consequently,
\[
Q\asymp\sum_{\ell=0}^{2m}2^{(\ell+\floor{\ell/2})t}x^\ell\pmod{x^{2m+1}}.
\]

Hence, we have
\begin{align*}
RI(T_{3,t})^m (1+x)&\asymp \left(\sum_{k=0}^\infty t^k x^k\right)\left( \sum_{k=0}^\infty 2^{(k+\floor{k/2})t}x^k\right) \\
&\asymp \sum_{k=0}^\infty \sum_{j=0}^k t^{k-j}2^{(j+\floor{{j}/{2}})t}x^k\\
&\asymp \sum_{k=0}^{\infty} 2^{(k+\floor{k/2})t}x^k \pmod{x^{2m+1}}.
\end{align*}

\item From \eqref{l:ii1} and proceeding similarly to \eqref{l:v}, we obtain
\begin{align*}
RI(S_{2,t})^{3m}&\asymp \sum_{k=0}^{\infty}\sum_{j=0}^m 2^{(3j+k)t}x^{3j+k}\asymp\sum_{k=0}^\infty 2^{kt}x^k\pmod{x^{2m+1}}.
\end{align*}\end{enumerate}	
\end{proof}

\begin{lemma}\label{l:coeffTG} Let $m$ be a fixed positive integer and let $t$ be a positive integer variable. Then:
\[
RI(TG_{m,t})\asymp \sum_{k=0}^{\infty}2^{(k+\floor{k/2})t} x^k\pmod{x^{2m+1}}
\]
\end{lemma}
\begin{proof} Lemma \ref{l:comp}\eqref{l:v}, Lemma \ref{l:comp}\eqref{l:vi}, \eqref{eq:main}, and Proposition~\ref{prop:0} yield the result.
\end{proof}

This analysis reveals how the log-concavity  of $I(TG_{m,t})$ is broken multiple times.

\begin{theorem}\label{thm:main}
  For each positive integer $m$ and for every sufficiently large integer $t$, the log-concavity of the independence polynomial of the tree $TG_{m,t}$ is broken at $m$ indices.
\end{theorem}
\begin{proof}
Let $0 \leq k \leq 2(m-1)$. According to Lemma~\ref{l:coeffTG}, $c_k(m,t)\in\Theta(2^{u_kt})$, where $u_k=k+\floor{k/2}$. Let us assume that $k$ is even and write $k=2j$. Then, $u_k = 3j$, $u_{k+1} = 3j + 1$, $u_{k+2}=3(j+1)$, and 
\[
c_k(m,t) c_{k+2}(m,t) \in \Theta(2^{(6j+3)t}), \quad c_{k+1}^2(m,t) \in \Theta(2^{(6j+2)t}),
\]
which implies
\[
c_k(m,t)c_{k+2}(m,t) - c_{k+1}^2(m,t) \in \Theta(2^{(6j+3)t}).
\]
\end{proof}

 For example, straightforward calculations using \eqref{eq:ITG} and \eqref{eq:ST} show that the graphs $TG_{4,6}$ and $TG_{5,6}$ break log-concavity exactly at indices $84, 82, 80, 78$ and $105, 103,$ $101, 99, 97$, respectively.

\bibliographystyle{amsplain} 
\bibliography{Bib}

\end{document}